\documentclass[reqno]{amsart}
\newcommand \datum {July 21, 2021}
\newcommand \addbibl{July 19, 2021}

\usepackage{amssymb,latexsym}
\usepackage{amsmath}
\usepackage{amscd}
\usepackage{graphicx}
\usepackage{color}
\usepackage[dvipsnames]{xcolor} 
\usepackage{multibib} 
\usepackage{enumerate}
\numberwithin{equation}{section}
\theoremstyle{plain}
 \newtheorem{theorem}{Theorem}[section]
 \newtheorem{lemma}[theorem]{Lemma}
 \newtheorem{proposition}[theorem]{Proposition}
 
\theoremstyle{definition}
 \newtheorem{definition}[theorem]{Definition}

 \newtheorem{remark}[theorem]{Remark}

\newcommand \At [1] {\textup{At}(#1)}
\newcommand \fprec {\mathrel{\prec_{\alg F}}}
\newcommand \gprec {\mathrel{\prec_{\alg G}}}
\newcommand \jtr [1]{f^\vee_{#1}}
\newcommand \mtr [1]{f^\wedge_{#1}}

\newcommand \ogmap {\overline\mu}
\newcommand \gmap {\mu}
\newcommand \lmap {\psi}
\newcommand \swedge[1]{\mathop{\wedge\kern-2pt{}_#1}}
\newcommand \svee[1]{\mathop{\vee\kern-2pt{}_#1}}
\newcommand \sbigwedge[1]{\bigwedge\kern-2pt{}_#1}
\newcommand \sbigvee[1]{\bigvee\kern-3pt{}_#1}

\newcommand \nothing [1] {}
\newcommand \dupl[1] {(the same as \cite{#1})}

\newcommand \con {\textup{con}}

\newcommand \defiff {\overset{\textup{def}}\iff}
\newcommand \nonparallel {\mathrel{\not{\kern-1.4pt{\mathord\parallel}} }}
\newcommand \wha [1] {{\alg{#1}^\ast}}
\newcommand \Lat [1]{\textup{Lat}(#1)}
\newcommand \Geom [1] {\textup{Geom}(#1)}
\newcommand \mclf {\textup{cl}_{\alg F}}
\newcommand \clf[1] {\textup{cl}_{\alg F}(#1)}
\newcommand \clFL[1] {\textup{cl}_{\alg F_L}(#1)}
\newcommand \cll[1] {\textup{cl}_{L}(#1)}

\newcommand \dwn {\mathord{\downarrow}}
\newcommand \Pow[1]{\textup{Pow}(#1)}

\newcommand \width[1]{\textup{width}(#1)}
\newcommand \then {\Rightarrow}
\newcommand \tbf[1]  {\textbf{#1}} 
\newcommand \alg[1]  {\mathcal #1}

\newcommand \Nnul {\mathbb N_0}
\newcommand \Jir [1] {\textup J(#1)} 
\newcommand \Mir [1] {\textup M(#1)} 
\newcommand \Nplu {\mathbb N^+}
\newcommand \set [1]{\{#1\}}
\newcommand \id [1] {\textup{id}_{#1}}

\renewcommand \phi{\varphi}

\newcommand \restrict [2] {#1\rceil_{\kern -1pt #2}}

\newcommand \Con [1]   {\textup{Con}(#1)}

\newcommand \ideal [1] {\mathord{\downarrow}#1}
\newcommand \filter [1] {\mathord{\uparrow}#1}
\newcommand \sideal [2] {\mathord{\downarrow}_{\kern-1pt#1}\kern 1pt #2}
\newcommand \sfilter [2] {\mathord{\uparrow}_{\kern-1pt#1}\kern 1pt#2}
\newcommand \sfolter [2] {\mathord{\Uparrow}_{\kern-1pt#1}\kern 1pt #2}
\newcommand \odeal [1] {\mathord{\Downarrow}#1}
\newcommand \sodeal [2] {\mathord{\Downarrow}_{\kern-1pt#1}\kern 1pt #2}

\newcommand \length [1]   {\textup{length}(#1)}

\newcommand \optional [1] {{#1}}

\newcommand\red[1]{{\textcolor{red}{#1}}}

\begin{document}
\title[Faigle geometries and semimodular lattices]
{Revisiting Faigle geometries from a perspective of semimodular lattices}

\author[G.\ Cz\'edli]{G\'abor Cz\'edli}
\email{czedli@math.u-szeged.hu}
\urladdr{http://www.math.u-szeged.hu/~czedli/}
\address{ Bolyai Institute, University of Szeged, Hungary}

\begin{abstract}  
In 1980, U. Faigle introduced  a sort of finite geometries on posets that are in bijective correspondence with finite semimodular lattices. His result has almost been forgotten in lattice theory. Here we simplify the axiomatization of these geometries, which we call \emph{Faigle geometries}. 
To exemplify their usefulness, we give a \emph{short} proof of a theorem of  Gr\"atzer and E. Knapp (2009) asserting that each slim semimodular lattice $L$ has a congruence-preserving extension to a slim rectangular lattice of the same length as $L$. As another application of Faigle geometries, we give a short proof of G.\ Gr\"atzer and E.\,W.\ Kiss' result from 1986 (also proved by M.\ Wild in 1993, the present author and E.\,T.\ Schmidt in 2010, and B. Skublics in 2013) that each finite semimodular lattice $L$ has an extension to a geometric lattice of the same length as $L$.
\end{abstract}

\thanks{This research of the first author was supported by the National Research, Development and Innovation Fund of Hungary under funding scheme K 134851.}



\date{\datum\hfill{{Hint: check the author's website for preprints and possible updates}}}

\maketitle

\section{Introduction, goal, and motivation}\label{sect:intro}

Although we postpone some of the necessary definitions to later sections, note that this paper is (intended to be) self-contained for all readers familiar with the rudiments of lattice theory, and  only few new concepts will be introduced or recalled.    

The antecedents of the paper belong to two categories.
First, in 1980, Faigle \cite{faigle} introduced several versions of geometries on finite posets (that is, on finite partially ordered sets) including pregeometries, geometries, and proper geometries. 
His proper geometries, which we call \emph{Faigle geometries}, are in bijective correspondence with finite semimodular lattices.

Second, we mention two famous results on length-preserving embeddings of semimodular lattices. The first  is due to  Gr\"atzer and Kiss~\cite[Lemma 17]{ggkiss}, 1986, and it was also proved by Wild~\cite[Theorem 4]{wild}, Cz\'edli and Schmidt~\cite{czgscht2geom}, and Skublics~\cite{skublics}. The second result was proved by Gr\"atzer and Knapp \cite{gratzerknapp3} in 2009.

\begin{theorem}[Gr\"atzer and Kiss~\cite{ggkiss}]\label{thm:ggkiss}
Each finite semimodular lattice has a length-preserving embedding into a finite geometric lattice.
\end{theorem}

\begin{theorem}[Gr\"atzer and Knapp \cite{gratzerknapp3}]\label{thm:gknapp}
For each slim semimodular lattice $L$ with at least three elements  there exists a congruence-preserving extension $K$ of $L$ such that $K$ is a slim rectangular lattice and it is of the same length as $L$.
\end{theorem}

\subsection*{Goal} 
We simplify the axiomatization of Faigle geometries; see Definitions
~\ref{def:lFglG} and  \ref{def:fgMtry}. 
We recall a lemma from Wild~\cite{wild} and give a simple proof for a lemma taken from Cz\'edli~\cite{czgreprhomo98}. 
Finally, as an application of Faigle geometries and the two lemmas just mentioned, we present 
a new and \emph{short} proof for each of Theorems~\ref{thm:ggkiss} and  \ref{thm:gknapp}.

\subsection*{Outline} The rest of this section gives  more details of our motivation together with some bibliographic references. Section~\ref{sect:faigle} contains our definitions and view of Faigle geometries in a self-contained way; see Definitions~\ref{def:lFglG} and \ref{def:fgMtry},  Theorem~\ref{thmfaigle}, and Lemma~\ref{lemma:smhDf}. Section~\ref{sect:2lemmas} states Lemmas~\ref{lemma:nvRgJzsj} and \ref{lemma:corner} and proves the second one. 
Sections~\ref{sect:glats} and \ref{sect:defproof} define some  concepts and prove  Theorems~\ref{thm:ggkiss} and \ref{thm:gknapp}, respectively.

\subsection*{More about our motivation}
The bijective correspondence between Faigle geometries and finite semimodular lattices is important. I fully agree with Quackenbush~\cite{qbushMR} that Faigle's ``work should prove to be useful in
the analysis of semimodular lattices.'' However, in spite of Quackenbush's initiative,  the above-mentioned bijective correspondence has hardly been exploited in lattice theory and it has almost completely been forgotten by now. 
This is indicated by a \addbibl{} MathSciNet search ``Anywhere=(Faigle and semimodular and geometry)'', which only returned four matches (not counting those two where ``Faigle'' only occurs as the reviewer), and all these four matches are from the period 1980--1986. As a possible reason, we mention that while defining several variants of combinatorial structures, Faigle~\cite{faigle} did not simplify the definition of those structures, the Faigle geometries, that we need here.

As we have already mentioned, Wild~\cite[Theorem 4]{wild} also proves Theorem~\ref{thm:ggkiss}. His proof is short but assumes  familiarity with matroid theory. 
As opposed to matroids, finite semimodular lattices are in bijective correspondence with Faigle geometries. Hence, it seems natural to give a \emph{short and self-contained} proof based on these geometries; we do so in Section~\ref{sect:glats}. 
It is worth mentioning that, as Wild~\cite{wild} points out, even a 1973 construction by Crawley and Dilworth~\cite[Theorem 14.1]{crawleydilworth} yields Theorem~\ref{thm:ggkiss}; see also the  historical comments in Cz\'edli and Schmidt~\cite[Section 4]{czgscht2geom} and in Skublics~\cite[Section 1]{skublics}. Note that
Skublics~\cite[Corollary 2]{skublics} gives a highly nontrivial generalization of Theorem~\ref{thm:ggkiss} for infinite lattices; his  proof is long and involved.
{
Note that Theorem~\ref{thm:ggkiss} is only a consequence of stronger results proved in Cz\'edli and Schmidt~\cite{czgscht2geom},  Gr\"atzer and Kiss~\cite{ggkiss}, and Skublics~{skublics}, but 
the proofs given there would remain long and involved even if they were tailored to Theorem~\ref{thm:ggkiss}.}

Slim semimodular lattice were introduced by Gr\"atzer and Knapp~\cite{gratzerknapp1} in 2007. These lattices are finite and necessarily \emph{planar}. Four dozen papers (including the present one) have been devoted to these lattices and their applications since then; see Cz\'edli and Gr\"atzer~\cite{czgggltsta}, Cz\'edli and Kurusa~\cite{czgkurusa}, the ``mini-survey'' subsection of Cz\'edli~\cite{czg-slimpatchabsretr}, and their references for most of these four dozen papers.
\optional{See also the Appendix (Section~\ref{sect:addbibl}) or 
\texttt{www.math.u-szeged.hu/\textasciitilde{}czedli/m/listak/publ-psml.pdf}.}
Slim rectangular lattices were also introduced by Gr\"atzer and Knapp but in another paper, \cite{gratzerknapp3}. These lattices play a central role in the theory of planar semimodular lattices; 
partly because of Theorem~\ref{thm:gknapp}.

The original proof of Theorem \ref{thm:gknapp} as well as those of many other results on   planar semimodular lattices are \emph{visual}. The advantage of this visual feature  is that lots of results on  planar semimodular lattices have been found  in a short time. However, there is some disadvantage, too:  visual proofs rely on many earlier results and the reader often has to look into many earlier papers, including Kelly and Rival's fundamental  \cite{kellyrival}, if he wants to \emph{really verify} these proofs.
The  proof of Theorem \ref{thm:gknapp} here is algebraic, easy to verify, shorter than the original one, and we present it in a self-contained paper. 
Not relying on geometric intuition, it might be easier to generalize the present proof for higher dimensions in the future than the earlier one based on planar geometrical tools.

\section{Faigle geometries}\label{sect:faigle}
Let $P$ be a poset. If $X\subseteq P$ such that for any $x\in X$ and $y\in P$, $y\leq x\then y\in X$, then $X$ is a \emph{down-set} of $P$. As  usual, for $u\in P$, the \emph{$($principal$)$ down-set}  $\set{x\in L: x\leq u}$ and the  the \emph{$($principal$)$ up-set} $\set{x\in L: x\geq u}$ are denoted by $\ideal u=\sideal P u$ and $\filter u=\sfilter P u$, respectively. Also, we denote $\sideal P u\setminus\set u$ by $\sodeal P u$ or, if no ambiguity threatens, by $\odeal u$. Thinking of the Hasse diagram of $P$, the notations 
\[ \sodeal P u := \set {x\in P:  x< u},\quad    \sideal P u := \set {x\in P:  x\leq u},\quad \sfolter P u := \set {x\in P:  x> u}
\]
are quite visual; e.g., both $\sideal P u$ and $\sodeal P u$ consist of  elements below or  equal to $u$ but the \emph{double} arrow reminds us that $\sodeal P u$ consists of elements \emph{strictly} below $u$.
Note that it will frequently occur that $u$ belongs to several posets; then the subscript of the vertical arrow is vital to make it clear that, say,  $\sideal B u$ is a subset of $B$.

Next, for a set $A$, a subset $\alg T$ of the \emph{power set} $\Pow A=\set{X: X\subseteq A}$ of  $A$ is a \emph{closure system on $A$} if $A\in\alg T$ and $\alg T$ 
is closed with respect to arbitrary (not only finitary) intersections. Closure systems are sometimes called \emph{Moore families}. With respect to ``$\subseteq$'', they are well known to be complete lattices. For $X,Y\in \alg T$, we denote by $X\prec Y$ or $X\prec_{\alg T} Y$ that $Y$ covers $X$ in $\alg T$, that is,
$X\subset Y$ but there is no $Z\in \alg T$ such that $X\subset Z\subset Y$. (As it is usual in  lattice theory, ``$X_1\subset X_2$'' means the conjunction of  ``$X_1\subseteq X_2$'' and  ``$X_1\neq X_2$''.)
Our first definition of Faigle geometries is quite simple. Since this is what we can conveniently use when studying finite semimodular lattices, we take the liberty to call it ``lattice theoretical''.

\begin{definition}[A lattice theoretical definition of Faigle geometries]\label{def:lFglG}
Given a finite poset $P$ and a subset $\alg F$ of the power set $\Pow P=\set{X: X\subseteq P}$ of  $P$,  the structure $(P,\alg F)$ is a \emph{Faigle Geometry} if
\begin{enumerate}
\item[(F$\cap$)] $P\in\alg F$ and  $\alg F$ is $\cap$-closed, that is, for all $X,Y\in\alg F$, we have that $X\cap Y\in\alg F$; 
\item[(F$\dwn$)] every member of $\alg F$ is a down-set of $P$;
\item[(Pr)] $\emptyset\in \alg F$ and for each $u\in P$, both $\sideal P u$ and $\sodeal P u $ belong to $\alg F$; and
\item[(CP)] for any $u\in P$ and $X\in \alg F$ such that $u\notin X$ and $\sodeal P u\subseteq X$, there exists a $Y\in \alg F$ such that $X\fprec Y$ and $u\in Y$.
\end{enumerate} 
In other words, by a Faigle geometry we mean a pair  $(P,\alg F)$ of a finite poset $P$ and a closure system $\alg F$ on $P$ satisfying (F$\dwn$),  (Pr), and (CP). 
\end{definition}

\begin{remark}\label{rem:uniqY}
 Note that ``there exists a $Y\in \alg F$'' in (CP)
can be replaced by ``there exists a unique $Y\in \alg F$''.
\end{remark}

\optional{
To see the validity of Remark~\ref{rem:uniqY}, observe that if we had distinct $Y_1$ and $Y_2$ satisfying the requirements of (CP), then $u\in Y_1\cap Y_2=X$ would be a contradiction.}

Geometries as combinatorial structures are usually defined with the help of closure operators. Hence, recall that a map $f\colon \Pow A\to\Pow A$ is a \emph{closure operator on $A$} if $X\subseteq Y\in\Pow A$ implies that $X\subseteq f(X)\subseteq f(Y)=f(f(Y))$.
There is a well-known bijective correspondence between the set of closure operators on $A$ and the set of  closure systems on $A$. Namely,
the closure operator associated with a closure system $\alg T$ is defined by $X\mapsto \bigcap\set{Y\in\alg T: X\subseteq X}$ while the
closure system corresponding to a closure operator $f$ is $\set{X\in \Pow A:  f(X)=X}$.

\begin{definition}[Second definition of Faigle geometries]\label{def:fgMtry}
Given a finite poset $P$ and a set  $\alg F\subseteq \Pow P$,  the structure $(P,\alg F)$ is a \emph{Faigle geometry} if, with the notation
\begin{equation}
\mclf\colon \Pow P\to\Pow P, \quad\text{defined by}\quad X\mapsto \bigcap\set{Y\in\alg F: X\subseteq Y},
\end{equation}
$(P,\alg F)$ satisfies  (F$\cap$), (F$\dwn$),  (Pr), and 
\begin{enumerate}
\item[(FEP)] for any  $u,v\in P$ and $S\in \alg F$, if $u\notin S$, $v\notin S$, $\odeal u \subseteq S$, and $v\in\clf{S\cup\set u}$,   then $u\in\clf{S\cup\set v}$.
\end{enumerate} 
\end{definition}

We are going to show that Definitions~\ref{def:lFglG} and \ref{def:fgMtry} are equivalent and, furthermore, each of these two definitions defines what Faigle~\cite{faigle} called ``proper geometries''. With reference to a more general class of geometries (and pregeometries), Faigle~\cite{faigle} defined ``proper geometries'' in a more complicated way. We prefer to call his 
``proper geometries'' as \emph{Faigle geometries} since, from the perspective of lattice theory, we consider them the most important structures defined in Faigle~\cite{faigle}.

Instead of recalling Faigle's definition and explaining directly why ours is equivalent to his,
we are going to restate and prove his theorem as Theorem~\ref{thmfaigle} based on Definition~\ref{def:fgMtry}. The reasons of this strategy are the following. 
First, this proof makes the paper self-contained and easier to read, especially if the reader wants to read the proof based on \emph{our} definition.
Second, some notations occurring in this theorem are needed later.
Third, it will automatically follow from Theorem~\ref{thmfaigle}  that Definition~\ref{def:fgMtry} is equivalent to the one in Faigle~\cite{faigle}. 
Fourth, the proof is not very long. 
\optional{Even if we 
presented the  original definition from Faigle~\cite{faigle} and showed why the two definitions are equivalent rather than giving a proof here, then we could save not more than a single page while creating some inconvenience for the reader.} 
Fifth, armed with Theorem~\ref{thmfaigle}, it will be easy to show that Definitions~\ref{def:lFglG} and \ref{def:fgMtry} are equivalent.%

Note that Stern~\cite[page 234]{stern} gives an account of what  Faigle~\cite{faigle} has done and contains some historical comments. 
Note also that the notation (FEP), which was denoted by (GEP) in Faigle~\cite{faigle}, comes from ``Faigle Exchange Property''. The acronym (CP) comes from ``covering property" while (Pr) about principal down-sets and their ``beheaded versions'' reminds us to ``principal''. Note also that, trivially, we can omit the stipulation ``$u\notin S$'' from (FEP) without changing the concept determined by Definition~\ref{def:fgMtry}.

A  lattice $L$ is (upper) \emph{semimodular} if, for all $x,y\in L$,  $x\wedge y\prec x \then y\prec x\vee y$. The  poset of nonzero join-irreducible elements of $L$ will be denoted by $\Jir L$.

\begin{definition}\label{def:fGccRspd}
For a finite semimodular lattice $L$, we define the \emph{Faigle geometry associated with} $L$ as
\begin{equation}\Geom L:=(\Jir L, \set{\Jir L\cap \ideal x: x\in L}); \text{ we also denote it by }(P_L,\alg F_L).
\label{eq:mtcnzsVMfgTpnLnH}
\end{equation}
To ease the notation,  $\clFL X$ in the sense of (FEP) will be written as  $\cll X$.
For a Faigle geometry $\wha F:=(P,\alg F)$ in the sense of Definition~\ref{def:fgMtry}, we define the \emph{lattice associated with} $\wha F$,  also called the \emph{lattice of flats} of $\wha F$,  as
\begin{equation} \Lat {\wha F}:=(\alg F,\subseteq).
\label{eq:ksmbmghgnKszGTrz}
\end{equation}
For Faigle geometries (no matter in which sense) $\wha F:=(P,\alg F)$ and $\wha {F'}:=(P',\alg F')$, we say that these geometries are \emph{isomorphic}, in notation $\wha F\cong\wha {F'}$, 
if there is a poset isomorphism $\gmap\colon P\to P'$ such that $\alg F'=\set{\ogmap(X): X\in \alg F}$ where 
$\ogmap$ is the map $\Pow P\to\Pow {P'}$ defined by 
$\ogmap(X)=\set{\gmap(p): p\in X}$. If there is such a $\gmap$, then $\ogmap$ and $\gmap$ determine each other and we use  the terminology
that $\ogmap\colon\wha F\to \wha{F'}$ is an \emph{isomorphism}.
\end{definition}

According to the following theorem, 
Faigle geometries and finite semimodular lattices are different
faces of the same entities in the following canonical way. 
Apart from slight differences in definitions, the following result is due to Faigle~\cite{faigle}.

\begin{theorem}[{Faigle \cite[Thm.\ 1(c) and Lemma 1]{faigle}}]\label{thmfaigle}
 If  $L$ is a finite semimodular lattice and $\wha F:=(P,\alg F)$ is a Faigle geometry in the sense of Definition~\ref{def:fgMtry}, then the following assertions hold.

\textup{(A)} $\Geom L$ is a Faigle geometry in the sense of Definition~\ref{def:fgMtry}. Also, for $X\subseteq 
P_L=\Jir L$, we have that  $\cll X=\Jir L\cap\sideal L {\sbigvee L X}$.

\textup{(B)} $\Lat {\wha F}$ is a finite semimodular lattice.
For elements $X$ and $Y$ of this lattice, that is, for $X,Y\in\alg F$, we have that $X\wedge Y=X\cap Y$, $X\vee Y=\clf{X\cup Y}$, and 
\begin{equation}
X\prec Y\,\, \iff \,\,(\exists u\in Y\setminus X)\,(\sodeal P u \subseteq X\text{ and }Y=\clf{X\cup\set u}.
\label{eq:thmBhLtt}
\end{equation}

\textup{(C)}  $\Lat{\Geom L}\cong L$ and the map $\lmap\colon L\to \Lat{\Geom L}$ defined by $\lmap(x):=\Jir L\cap \sideal L x$ is an isomorphism.

\textup{(D)} $\Geom{\Lat {\wha F}}\cong \wha F$ and, with
$\wha{F'}=(P',\alg F'):= \Geom{\Lat{\wha F}}$ and 
$\gmap\colon P\mapsto P'$ defined by $\gmap(u)=\sideal P u$, 
$\,\,\ogmap\colon \wha F\to \wha{F'}$ is an isomorphism.
\end{theorem}

Note that even if the formalism here is different, most steps of the proof  below can be found in Faigle~\cite{faigle}

\begin{proof}[Proof of Theorem~\ref{thmfaigle}]
To prove (A), let $L$ be a finite semimodular lattice, let $P$ be the poset $\Jir L$, and let $\alg F:=\set{\Jir L\cap \sideal L x: x\in L}$. We are going to understand $\vee$, $\bigvee$, and  $\wedge$ in $L$. 
With the possible exception of $\sodeal P u \in\alg F$, observe that 
(F$\dwn$), (F$\cap$), (Pr), and the description of 
$\cll X$  trivially hold.
Let $u\in P=\Jir L$ and $x_u:=\bigvee  \sodeal P u$. 
Then $x_u\leq u$ since $v < u$ for every joinand $v\in \sodeal P u$. 
In fact, $x_u<u$ since $u$ is join-irreducible. If $v\in \Jir L\cap\sideal L {x_u}$, then $v\in \sodeal P u$ since $v\leq_L x_u <_L u$. Conversely, if
$v\in\sodeal P u$, then $v\in \Jir L\cap\sideal L {x_u}$ since $v$ is a joinand of $x_u$. Hence, 
\begin{equation}\sodeal P u=\Jir L\cap\sideal L {x_u}\in \alg F, 
\label{eq:msGszhkSzrgmLSzR}
\end{equation}
as required. 
Next, to prove that $\Geom L$ satisfies (FEP), 
assume that $S=P\cap \sideal L s\in \alg F$, $u,v\in P\setminus S$ (whence $u,v\not\leq_L s$), 
$\sodeal P u\subseteq S$, and $v\in\clf{S\cup \set u}$.
By \eqref{eq:msGszhkSzrgmLSzR}, $x_u\leq s$. Since each element $y$  of $L$ is $\bigvee  (P\cap \sideal L y)$, it follows trivially that $x_u\prec_L u$. Hence, semimodularity and $u\not\leq s$ give  that $s=s\vee  x_u\prec_L s\vee  u$. Since $v\in\cll{S\cup \set u}$, the description of $\cll X$ gives that $v\leq_L \bigvee (S\cup \set u)=s\vee  u$. But $v\not\leq_L s$, so $s<_L s\vee  v\leq s\vee  u$. Combining this with $s \prec_L s\vee  u$, we obtain that $u\leq s\vee  u=s\vee  v$, whereby $u\in\cll{S\cup\set v}$, as (FEP) requires,  proving (A).

To prove (B), recall from the folklore that the members of a closure system $\alg F$ always form a (complete) lattice in which $X\wedge Y=X\cap Y$, $X\vee Y=\clf{X\cup Y}$, and $X\leq Y\iff X\subseteq Y$. In particular,  $\Lat{\wha F}$
is a lattice. 
Assume that $X\prec Y$ in  $\Lat{\wha F}$. With respect to $\leq_P$, take a \emph{minimal} element $u$ of $Y\setminus X$.
Clearly, $\sodeal P u\subseteq X$ and $X < \clf{X\cup \set u}\leq Y$, whence $Y$, which covers $X$, is $\clf{X\cup \set u}$, as required. 
Conversely, assume that $X\in \Lat{\wha F}$, that is, $X\in\alg F$, 
$\sodeal P u\subseteq X$, $u\notin X$, and $Y=\clf{X\cup\set u}$. 
Clearly, $X<Y$. Assume that $Z\in\alg F$ and $X<Z\leq Y$.
Take an arbitrary $v\in Z\setminus X$. Since $v\in Y=\clf{X\cup\set u}$, (FEP) implies that $u\in\clf{X\cup\set v}$. Hence,
$Y=\clf{X\cup\set u}\leq \clf{\clf{X\cup\set v}}=\clf{X\cup\set v}\leq Z$ yields that $Z=Y$, whereby $X\prec Y$ as required. This verifies \eqref{eq:thmBhLtt}.
To prove that $\Lat{\wha F}$ is semimodular, assume that 
$X,Y,Z,U\in \alg F$ such that $Z=X\wedge Y\prec X$ and $U=X\vee Y$.
We need to show that $Y\prec U$. Clearly, we can  assume that $X$ and $Y$ are  incomparable, in notation $X\parallel Y$.
Since $Z\prec X$, \eqref{eq:thmBhLtt} allows us to pick a $q\in X\setminus Z$ such that $\sodeal P q \subseteq Z$ and $X=\clf{Z\cup \set q}$. Then $Z\subseteq Y$ gives that  $\sodeal P q \subseteq Y$ and $U=X\vee Y=\clf{X\cup Y}= \clf{\clf{Z\cup \set q} \cup Y} = \clf{Z\cup \set q \cup Y} = \clf{Y\cup \set q}$. 
We also have that
$q\notin Y$ since otherwise $X=\clf{Z\cup\set q}\subseteq Y$ would contradict that $X\parallel Y$. 
Hence,  \eqref{eq:thmBhLtt} gives that  $Y\prec U$, as required. Thus, $\Lat{\wha F}$ is semimodular, proving (B).

Next, in addition to $\psi$ defined in part (C),  we take the map \[
\psi_1\colon \Lat{\Geom L}\to L\text{ defined by }Y \mapsto \sbigvee L{Y}.
\]
Clearly, both $\psi$ and $\psi_1$ are order-preserving. Since 
each  $x\in L$  is the join of $\Jir L\cap\sideal L x$, we have that 
$\psi_1\circ\psi=\id L$, the \emph{identity map} $L\to L$ defined by $y\mapsto y$. Note that we compose maps from right to left, that is, 
$(\psi_1\circ \psi)(y)=(\psi_1(\psi(y))$. The equality $\psi\circ\psi_1=\id{\Lat{\Geom L}}$ is also easy since any $Y\in  \Lat{\Geom L}$ is of the form $Y=\Jir L\cap\sideal L x$, whence
$\psi_1(Y)=x$ and $\psi(x)=Y$.  Thus $\psi$ and $\psi_1$ are reciprocal isomorphisms, proving part (C).

To prove part (D), we let $K:=\Lat{\wha F}$. Then $P':=\Jir K$, 
$\alg F':=\set{P'\cap \sideal K X: X\in K\text{, that is, }X\in \alg F}$, and $\wha {F'}:=\Geom K$. Then $\Geom{\Lat{\wha F}}=\wha{F'}$. 
We know from (Pr) that $\sodeal P u\in \alg F=K$ for every  $u\in P$. Clearly, $\sodeal P u$ is the only lower cover of $\sideal P u$ in $K$, whence $\gmap(u)=\sideal P u\in\Jir K=:P'$. 
Hence, $\gmap$ is a $P\to P'$ map indeed, and it is clearly order-preserving.  Let $X\in K$, that is, $X\in \alg F$. Using that $X$ is a down-set in $P$ at `` $=^\ast$ '' below,
\[\sbigvee K\set{\sideal P u: u\in X}=\clf{\bigcup\set{\sideal P u: u\in X}}=^\ast \clf X=X.
\]
Thus, each element $X$ of $K$ is the join of elements of the form $\sideal P u=\gmap(u)$. Hence, $\Jir K=\set{\sideal P u: u\in P}
=\set{\gmap(u): u\in P}$ and $\gmap$ is surjective. For $u,v\in P$, if $\gmap(u)=\gmap(v)$, then $u\in \sideal P u=\gmap(u)=\gmap(v)=\sideal P v$ shows that $u\leq_P v$. We obtain similarly that $v\leq _P u$. Thus, $\gmap$ is injective, so it is a bijection. If $\gmap(u)\leq_K\gmap (v)$, then $u\in \sideal P u\subseteq \sideal P v$ gives that $u\leq_P v$, whereby $\gmap^{-1}$ is order-preserving. Thus, $\gmap\colon P\to P'$ is an order isomorphism, and we can turn our attention to $\ogmap\colon \wha F\to \wha{F'}$.

Let $X\in\alg F$. For $u\in P$, $\gmap(u)=\sideal P u\leq_K X$ if and only if $u\in X$ since $X$ is a down-set in $P$. Using that $\Jir K=P'=\gmap(P)$, we obtain that $\ogmap(X)=\set{\gmap(u): u\in X}=
\Jir K\cap\sideal K X\in \alg{F'}$. Hence, $\set{\ogmap(X): X\in\alg F}\subseteq \alg{F'}$. To show the converse inclusion, let $Y\in \alg{F'}$. Then there is an $X\in K=\alg F$ such that 
$Y=\set{H\in \Jir K=P': H\leq_K X}=\set{\gmap(u): u\in P \text{ and } \gmap(u)\leq_K X}$. But ``$\leq_K$'' here means ``$\subseteq$'' and $\gmap(u)=\sideal P u\subseteq X \iff u\in X$ since $X$ is a down-set in $P$. Hence, $Y=\set{\gmap(u): u\in X}=\ogmap(X)$, showing that $\ogmap$ is an isomorphism. This completes the proof of (D) and that of Theorem~\ref{thmfaigle}
\end{proof}

To conclude this section, we formulate and prove the following easy lemma.

\begin{lemma}\label{lemma:smhDf} 
Definitions \ref{def:lFglG} and \ref{def:fgMtry} are equivalent.
\end{lemma}

\begin{proof} Let $P$ be a finite poset, $\alg F\subseteq \Pow P$, and assume that $\wha F:=(P,\alg F)$ satisfies (F$\cap$), (F$\dwn$), and (Pr). We need to show that with these assumptions, (FEP) and (CP) are equivalent for $\wha F$.
First, assume that (FEP) holds for $\wha F$. Let $u\in P$ and $X\in \alg F$ such that $u\notin X$ and $\sodeal P u\subseteq X$. 
It follows from \eqref{eq:thmBhLtt} that $Y:=\clf{X\cup\set{u}}$ covers $X$ and contains $u$. Hence, (CP) holds for $\wha F$. 

Second, assume that $\wha F$ satisfies (CP). Let  $u,v\in P$ and $S\in \alg F$ such that $u\notin S$, $v\notin S$, $\odeal u \subseteq S$, and $v\in\clf{S\cup\set u}$. By (CP) and Remark~\ref{rem:uniqY}, 
there is a unique $Y\in \alg F$ such that $u\in Y$ and $S\fprec Y$. 
Since $S\cup\set u\subseteq Y\in\alg F$ gives that
$S\subset \clf{S\cup\set u}\subseteq Y$, the covering $S\fprec Y$ implies that $\clf{S\cup\set u} = Y$. Similarly,  
$S\subset S\cup\set v \subseteq  \clf{S\cup\set u} = Y$ and
 $S\fprec Y$  yield that $\clf{S\cup\set v} = Y$. Thus,
$u\in \clf{S\cup\set u} = Y = \clf{S\cup\set v}$ shows that (FEP) holds for $\wha F$. 
\end{proof}

\begin{remark}\label{rem:nlkLbkSgd}
In what follows, Lemma~\ref{lemma:smhDf} allows us not to make a sharp distinction between Definitions~ \ref{def:lFglG} and \ref{def:fgMtry}. This means that if we are given a Faigle geometry, then we can use both (FEP) and (CP) without  separate explanation. Conversely, to prove that $(P,\alg F)$ is a Faigle geometry, 
it suffices to show that at least one of (FEP) and (CP) holds in addition to (F$\cap$), (F$\dwn$), and (Pr). 
Also, based on Theorem~\ref{thmfaigle}, we will often pass (implicitly sometimes) from  a Faigle geometry to the finite semimodular lattice associated with it or vice versa.
\end{remark}

\section{Two more lemmas}\label{sect:2lemmas}
The \emph{length} of a finite chain $C$ is $|C|-1$. The length  of a lattice $L$, denoted by $\length L$, is the supremum of $\{\length C : C$ is a finite chain in $L\}$. If $\length L\in\Nnul=\set 0\cup\Nplu$, then $L$ is of \emph{finite length}. If $L$ is a semimodular lattice of finite length, then $\length C=\length L$ for every maximal chain $C$ in $L$.  For a lattice $K$ and a nonempty subset $L\subseteq K$, if $L\subseteq K$ and $x\wedge_K y \in L$ for all $x,y\in L$, then 
$L$ is a \emph{meet-subsemilattice} of $K$. 
Note that $L$ can be a lattice with respect to the ordering inherited from $K$ even if $L$ is not a sublattice of $K$.
We need the following lemma, which is a particular case of Wild~\cite[Lemma 1]{wild}.

\begin{lemma}[Wild~\cite{wild}]\label{lemma:nvRgJzsj}
If $L$ and $K$ are semimodular lattices of the same finite length and $L$ is a meet-subsemilattice of $K$, then $L$ is a sublattice of $K$. 
\end{lemma}

\optional{Similarly to Wild~ \cite[Lemma 1]{wild}, we note that 
neither semimodularity nor the assumption that $\length L=\length K$ can be omitted from Lemma~\ref{lemma:nvRgJzsj}. To  exemplify this, let  $K_1$ be  the direct product of the two-element chain and the three-element chain, and omit the unique join-reducible coatom
of $K_1$ to obtain $L_1$. Also, let $L_2=\set{0_{K_2},a,b,1_{K_2}}$ where $a$ and $b$ are distinct atoms of the eight-element boolean lattice  $K_2$.} 
For the reader's convenience, we outline the proof of Lemma~\ref{lemma:nvRgJzsj}; see Wild~\cite{wild} for more details. We say that \emph{$L$ is a cover-preserving sublattice} of $K$ if $L$ is a sublattice and for any $x,y\in L$, $x\prec_L y\iff x\prec_K y$.

\begin{proof}[Outline of the proof of  Lemma~\ref{lemma:nvRgJzsj}]
If $x\parallel y$ in $L$ and $x\wedge y$ happens to be a lower cover of $x$ and $y$ then, both in $L$ and $K$, the join of $x$ and $y$ covers $x$ and $y$. Since  $L$ is a cover-preserving subposet of $K$, $x\vee_K y=x\vee_L y$. If, say, $x\wedge y\not\prec x$, then
for any $x'\in L$ with $x\wedge y<x'<x$, we have that 
$x\vee y=x\vee (x'\vee y)$. This allows us to use an induction on $\length{[x\wedge y,x]} + \length{[x\wedge y,y]}$.
\end{proof}

Given a lattice $L$ and $\rho\subseteq L\times L$, the congruence generated by $\rho$ will be denoted by $\con_L(\rho)$ or $\con(\rho)$. As usual, the \emph{congruence lattice} of $L$ is denoted by $\Con L$. If  $L$ is a sublattice of a lattice $K$, then 
\begin{equation}
\parbox{7.7cm}{ $\rho_{K,L}\colon \Con{K}\to\Con  {L}$ denotes the
\emph{restriction map} defined by $\Theta\mapsto \restrict\Theta {L}:=\Theta\cap(L\times L)$.
}\label{pbx:klBlmkPNsrKskjjB}
\end{equation}
Following Gr\"atzer and Schmidt~\cite{gr-scht}, $K$ is a \emph{congruence-preserving extension of $L$}
if $\rho_{K,L}$ is an order isomorphism (equivalently, a lattice isomorphism). Note that our definition is clearly equivalent to the one given in \cite{gr-scht}; yet another equivalent definition is that the \emph{extension map} $\Con L\to \Con K$ defined by  $\Theta\mapsto \con_K(\Theta)$ is a lattice isomorphism.

Dually to $\Jir L$, the poset of \emph{meet-irreducible} elements of a lattice $L$ is denoted by $\Mir L$. 
Although the following lemma is known from Cz\'edli~\cite[Lemma 5.4]{czgreprhomo98}, here we are going to prove it more simply and shortly than in  \cite{czgreprhomo98}.

\begin{lemma}[Corner Lemma from Cz\'edli~\cite{czgreprhomo98}]\label{lemma:corner} 
Let $L$ be a sublattice of a lattice $K$ of finite length, and let $a,b,c,d\in K$ such that 
$L=K\setminus\set d$, $a\prec_K c\prec_K b$,  $a\prec_K d\prec_K b$,
$a\in \Mir L$,  $b\in \Jir L$, and $d\in\Jir K\cap\Mir K$. Then $K$ is a congruence-preserving extension of $L$. 
\end{lemma}

\begin{proof}[Proof of Lemma \ref{lemma:corner}]
Let $S=\set{a,c,d,b}$; it is a 4-element boolean lattice with two ``old'' edges $a\prec c$ and $c\prec b$ and  two ``new edges'', $a\prec d$ and $d\prec b$. (Here $\prec$ stands for $\prec_K$; for elements in $L$ it is the same as $\prec_L$.) Let $I:=S\cap L=\set{a,c,b}$. Since each new edge is transposed to a (unique) old edge of $L$ and since any lattice congruence in a lattice of finite length is determined by the edges it collapses,  $\rho_{K,L}$ is injective. By the same reason, $S$ is a congruence-preserving extension of $I$.

Next, we define a map $\epsilon_{L,K}\colon\Con L\to \Con K$ as follows. Let $\beta\in\Con L$. Since $S$ is a congruence-preserving extension of $I$, the restriction $\restrict \beta I$ extends to a unique congruence $\gamma\in\Con S$. Let $\epsilon_{L,K}(\beta)$  be the transitive closure of $\beta\cup\gamma$. Clearly, both $\rho_{K,L}$ and  $\epsilon_{L,K}$ are order-preserving. It suffices to prove that  
\begin{equation}
\epsilon_{L,K}(\beta)\in\Con K\text{ and }
\rho_{K,L}(\epsilon_{L,K}(\beta))\subseteq \beta\text{ for every }\beta\in\Con L.
\label{eq:mNmrdjshshlkRkthNl}
\end{equation}
Indeed, the converse inclusion in \eqref{eq:mNmrdjshshlkRkthNl} is trivial, whence \eqref{eq:mNmrdjshshlkRkthNl} gives that $\rho_{K,L}$ is surjective, whereby it is an order isomorphism with inverse $\epsilon_{L,K}$. For convenience, we let $\delta:=\epsilon_{L,K}(\beta)$. For $u\in K$, let $\jtr u\colon K\to K$ and $\mtr u\colon K\to K$ be the maps defined by $x\mapsto u\vee x$ and $x\mapsto u\wedge x$, respectively.  For a relation $\mu\subseteq K\times K$, we let $\jtr u(\mu):=\set{(\jtr u(x),\jtr u(y)): (x,y)\in \mu}$; we define $\mtr u(\mu)$ similarly. We claim that 
\begin{equation}
\text{for every $u\in K$, $\,\, \jtr u(\beta\cup\gamma)\subseteq \delta$ and  $\mtr u(\beta\cup\gamma)\subseteq \delta$.}
\label{eq:mHrztKKRnsRpntl}
\end{equation}
By duality, it suffices to show $\jtr u(\beta)\subseteq \delta$
and $\jtr u(\gamma\setminus\delta)\subseteq \delta$.

First, let $(x,y)\in\beta$; in particular, $x,y\in L$. We can assume that $\jtr u(x)\neq \jtr u(y)$. We can also assume that  $u=d$ since otherwise $\beta\in\Con L$ implies that 
$(\jtr u(x),\jtr u(y))\in\beta\subseteq\delta$. Using the rule
$\jtr d(z)=\jtr d(\jtr a(z))$ and that  $\jtr a(\beta)\subseteq \beta$, we can assume that $x,y\in\filter_K a$. 
Since $b$ is the only cover of $d$ in $K$, $\sfolter K d=\sfilter K b=\sfilter L b$. Hence,
either $x=a$ and $\jtr d (x)=d$, or $\jtr d (x)=\jtr b(x)$, and similarly for $\jtr d(y)$. 
But $\jtr a(x)\neq  \jtr a(y)$, and $\jtr b(\beta)\subseteq\beta\subseteq\delta$ since $b\in L$.  Hence, apart from $x$--$y$ symmetry, we only need to deal with the case $x=a$ and $y\in \sfolter L a=\sfilter L c$; this last equality follows from $a\in\Mir L$.  
Since congruence blocks are convex sublattices, $(a,c)=(x,c)\in\beta$ and $(c,y)\in \beta$. Also, $(a,c)\in\gamma$ since  $\restrict \beta I=\gamma$. Thus, $(\jtr d(a),\jtr d(c))=(d,b)\in\gamma\subseteq \delta$. Also, $(\jtr d(c), \jtr d(y)) = (b, \jtr d(y))  =(\jtr b(c), \jtr b(y))\subseteq \beta\subseteq\delta$. 
Hence,  $(\jtr u(x), \jtr u(y))=(\jtr d(x), \jtr d(y))\in\delta$ since $\delta$ is transitive.

Second, let $(x,y)\in\gamma\setminus\beta$. Apart from $x$--$y$ symmetry, either $(x,y)=(a,d)$, or $(x,y)=(d,b)$. 
We can assume that $u\notin S$, since otherwise $\jtr u(\gamma)\subseteq \gamma\subseteq\delta$ applies, and that $u\notin \sideal L a$, since otherwise $(\jtr u(x),\jtr u(y))=(x,y)\in\gamma\subseteq\delta$. Then $\jtr u(z)=\jtr u(a\vee z)=\jtr {u\vee a}(z) =\jtr u(\jtr c(z))$ holds for all $z\in S$ since $c$ is the only cover of $a$ in $L$.  If $(x,y)=(a,d)$, then $(\jtr c(x),\jtr c(y))=(c,b)\in\restrict\gamma I=\restrict \beta I\subseteq \beta$, whereby $(\jtr u(x),\jtr u(y))= (\jtr u(\jtr c(x)),\jtr u(\jtr c(y)))\in\beta\subseteq\delta$ since $\jtr u(\beta)\subseteq\beta$. If $(x,y)=(d,b)$, then $(\jtr u(x),\jtr u(y))\in\delta$ since $\jtr u(x)= \jtr u(\jtr c(x))=b=\jtr u(\jtr c(y))=\jtr u(y)$. This proves \eqref{eq:mHrztKKRnsRpntl}, from which we immediately conclude that 
\begin{equation}
\text{for every $u\in K$, $\,\, \jtr u(\delta)\subseteq \delta$ and  $\mtr u(\delta)\subseteq \delta$.}
\label{eq:fBlPrtLdpncsfHl}
\end{equation}

Clearly, $\delta$ is an equivalence. Assume that $(x,y)$ and $(u,v)$ belong to $\delta$. Applying \eqref{eq:fBlPrtLdpncsfHl}, we obtain that
$(x\vee u, x\vee v)=(\jtr{x}(u),\jtr{x}(v)) \in\delta$ and  
 $(x\vee v, y\vee v) = (\jtr{v}(x),\jtr{v}(y))\in\delta$. Hence, transitivity yields that $(x\vee u, y\vee v)$ belongs to $\delta$. So does $(x\wedge u, y\wedge v)$ by duality, and we conclude that $\epsilon_{L,K}(\beta)=\delta \in \Con K$.

To prove the second half of \eqref{eq:mNmrdjshshlkRkthNl}, let 
$(x,y)\in\rho_{K,L}(\epsilon_{L,K}(\beta))$. That is, $x,y\in L$ and $(x,y)\in\delta$. Take a shortest sequence
$z_0=x, z_1,\dots,z_{k-1}, z_k=y$ such that for $i=1,\dots,k$,
either $(z_{i-1},z_i)\in\beta$ and we call the $i$-th step a \emph{$\beta$-step}, or $(z_{i-1},z_i)\in\gamma\setminus\beta$ and we call the $i$-th step a \emph{$\gamma$-step}. By the minimality of $k$, 
the elements $z_0,\dots,z_k$ are pairwise distinct. 
Since $d\notin L$ and so it cannot take part in a $\beta$-step,
every $\gamma$-step is followed or preceded by another $\gamma$-step. But two consecutive $\gamma$-steps can be replaced by a single $\beta$-step (from $a$ to $b$ or $b$ to $a$). Hence, the minimality of $k$ yields that there is no $\gamma$-step at all, whereby the transitivity of $\beta$ gives that $(x,y)\in\beta$, as required. Thus, \eqref{eq:mNmrdjshshlkRkthNl} holds, completing the proof of Lemma~\ref{lemma:corner}.
\end{proof}

\optional{
\section{Two illustrations}
We present two figures.
Since one of our aims is to give \emph{non-visual} short proofs in the paper, \emph{none of these figures is needed} in our proofs. However, the reader might be interested in the visual aspects of these proofs in case of two small examples given in Figure~\ref{fig1} for Section~\ref{sect:glats} and Figure~\ref{fig2} for Section~\ref{sect:defproof}. In both cases, the semimodularity of $L$ follows easily from, say, Cz\'edli and Schmidt~\cite{czgschtvisual}.
}

\optional{
\begin{figure}[h]
\centerline
{\includegraphics[width=\textwidth]{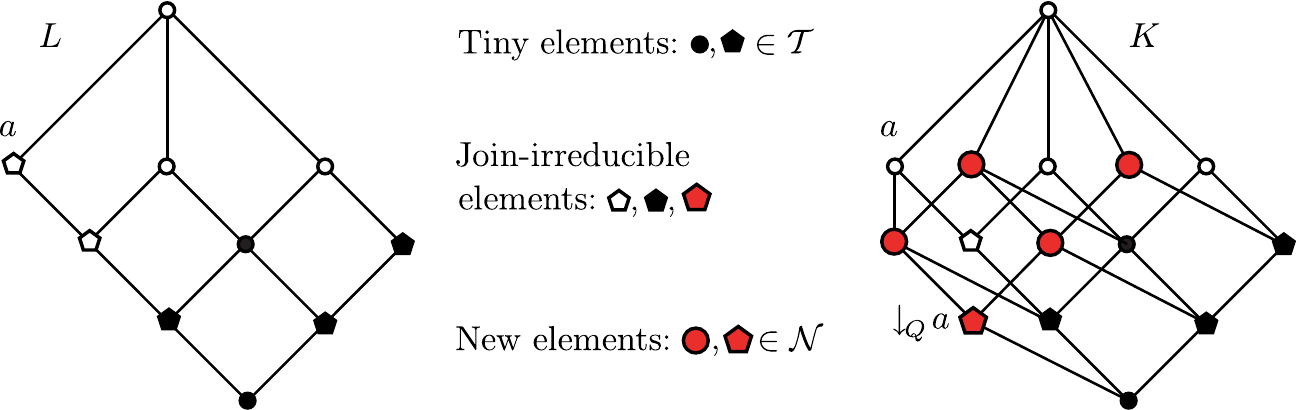}}      
\caption{An illustration for Section~\ref{sect:glats}}\label{fig1}
\end{figure}
}

\section{Extending semimodular lattices to geometric lattices}\label{sect:glats}
For a finite lattice $K$, let $\At K$ denote the \emph{set of atoms} of $K$, that is, the set  of covers of $0_K$.
A finite semimodular lattice $K$ is a \emph{geometric lattice} if $\Jir K=\At K$. Instead of Theorem~\ref{thm:ggkiss}, we are going to proof its stronger form, which was first proved  in Wild~\cite{wild} and was also proved in Cz\'edli and Schmidt~\cite{czgscht2geom}.
\optional{Since any two maximal chains of a finite semimodular lattice have the same number of elements, note that for semimodular lattices, length-preserving embeddings are the same as cover-preserving $\set{0,1}$-embeddings.}

\begin{proposition}[Wild~\cite{wild}]\label{prop:ggkss}
Each finite semimodular lattice $L$ has a length-preserving embedding into a finite geometric lattice $K$ such that $|\At K|=|\Jir L|$.
\end{proposition}

\begin{proof}
Assume that $L$ is not geometric. Pick a maximal element $a$  of $\Jir L\setminus \At L$. Let $\Geom L=\wha F=(P,\alg F)=(P_L,\alg F_L)=(\Jir L,\alg F_L)$ from \eqref{eq:mtcnzsVMfgTpnLnH}. Define a poset $Q=(Q,\leq_Q)$ with underlying set $Q:=P$ such that $x <_Q y\iff
(a\notin\set{x,y}\text{ and }x<_P y)$. 
Let $\alg T:=\{X\in \alg F: a\notin X$ and, for all $Y\in\alg F$,  $X\fprec Y\then a\notin Y\}$,  $\alg N:=\set{X\cup\set a: X\in\alg T}$, $\alg G:=\alg F\cup \alg N$, and $\wha G:=(Q,\alg G)$. We claim that
\begin{equation}
\parbox{8cm}{$\alg F\cap\alg N=\emptyset$. Also, if $X\in\alg N\cup\alg T$ and $X\supseteq Y\in \alg G$, then $Y$ is in $\alg N\cup\alg T$, and it is even in $\alg T$ if $X\in\alg T$.}
\label{eq:kMnTszLGnCskHvgrv}
\end{equation}
If $E=D\cup\set a\in\alg F\cap\alg N$, then $D\in\alg T$, $\sodeal P a\subseteq D$ by (F$\dwn$), and either $D\fprec \clf{D\cup\set a}\ni a$ by \eqref{eq:thmBhLtt} or   $a\in D$, contradicting $D\in \alg T$. Thus,  $\alg F\cap\alg N=\emptyset$. 
Assume that $Y'\subseteq X'\in \alg T$. Then $a\notin Y'$. If $a\in\clf{Y'\cup\set u}$ and $\sodeal P u \subseteq Y'$, then either $u\notin X'$ and \eqref{eq:thmBhLtt} gives that   $X'\fprec \clf{X'\cup\set u}\supseteq \clf{Y'\cup\set u}\ni a$, 
or $u\in X'$ and $a\in \clf{Y'\cup\set u}\subseteq \clf{X'}=X'$, contradicting $X'\in \alg T$. Hence, $Y'\in\alg T$.
If $X= X'\cup\set a \in \alg N$ and $X\supseteq Y'\in \alg F$, then
$a\notin Y'$ since otherwise $\sodeal P a\subseteq X'\in\alg T$ would contradict \eqref{eq:thmBhLtt} (with $a$ playing the role of $u$), whereby
$Y'\subseteq X'$ and $Y'\in\alg T$ as previously. Finally, if 
$X= X'\cup\set a \in \alg N$ and $X\supseteq Y= Y'\cup\set a\in\alg N$ then, again, $X'\supseteq Y'\in \alg T$, whence $Y\in\alg N$. Thus, \eqref{eq:kMnTszLGnCskHvgrv} holds. 
It follows easily from \eqref{eq:kMnTszLGnCskHvgrv} that  (F$\cap$) and (F$\dwn$) hold in $\wha G$. 
For $u\in Q\setminus\set a$,  
$\set{\sodeal Q u,\sideal P u}=\set{\sodeal P u,\sideal P u}\subseteq\alg F\subseteq \alg G$. Also,
$\sideal Q a=\emptyset \cup \set a\in \alg G$ and $\sodeal Q a=\emptyset \in\alg G$. Hence, (Pr) holds in $\wha G$. 

To prove that $\wha G$ satisfies (CP), assume that $D\in\alg G$, $u\in Q\setminus D$, and $\sodeal Q u\subseteq D$. We need to find an $E\in\alg G$ such that $D\gprec E\ni u$. 
First, let $u=a$. Then $D\in\alg F$ since $a=u\notin D$. If $D\in\alg T$, then we can let $E:=D\cup\set a$. Let $D\in\alg F\setminus\alg T$. Then 
$D\fprec\clf{D\cup\set a}=:E\ni a$ by \eqref{eq:thmBhLtt}, and  \eqref{eq:kMnTszLGnCskHvgrv} gives that $D\gprec E$. 

In the rest of the proof, we assume that $u\neq a$. Clearly, $a\notin\sideal P u$. 
 If $D\in\alg F$, then 
 $\sodeal P u=\sodeal Q u\subseteq D$, and
(CP) applied to $\wha F$ yields an $E\in\alg F$ such that  $D\fprec E\ni u$. If $D\in \alg F\setminus \alg T$, then  \eqref{eq:kMnTszLGnCskHvgrv} implies that $D\gprec E$, as required. If $D\in \alg T$, then $a\notin E$, whence no member of $\alg N$ is a subset of $E$, and $D\gprec E$ again, as required. 

Hence, we can assume that $D=X\cup\set{a}\in \alg N$ with $X\in\alg T$. Since $a\notin\sodeal P u$,
(CP) and  $\sodeal P u=\sodeal Q u \subseteq X$ yields a $Y\in \alg F$ with $X\fprec Y\ni u$. We can assume that $Y\notin\alg T$ since otherwise $D=X \cup\set{a}\gprec Y\cup\set{a}\in \alg N$ and $u\in Y\cup\set{a}$, as required. 
Since $Y\in\alg F\setminus \alg T$, there is an $E\in\alg F$ such that $Y\fprec E\ni a$. By $u\in Y$ and \eqref{eq:kMnTszLGnCskHvgrv},  $u\in E\supset D$. For the sake of contradiction, suppose that  $D\gprec E$ fails, and pick an $H\in \alg G$ such that $D\gprec H\subset E$. 
First, assume that $H\in \alg F$. Since $a\in D\subset H$ and $X\in \alg T$, we obtain that $X\not\fprec H$ while $X\subset D\subset H$ gives that $X<_{\alg F}H$. Hence, $X<_{\alg F} H'<_{\alg F} H<_{\alg F}E$ for some $H'\in \alg F$, contradicting $X\fprec Y\fprec E$ by the semimodularity of $L$. Hence, $H=Z\cup\set a\in\alg N$. Here $Z\in\alg T$, whence $Z\not\fprec E$. But $Z\subset E$, so we can pick a $Z'\in\alg F$ with $X<_{\alg F}Z<_{\alg F} Z'<_{\alg F}E$, which gives the previous contradiction. Thus, $D\gprec E$ and $\wha G$ satisfies (CP). Hence, $\wha G$ is a Faigle geometry. 
Let $K:=\Lat{\wha G}$; it is a semimodular lattice by Theorem~\ref{thmfaigle}(B). Clearly, $L$ is a meet-subsemilattice of $K$. Let $\Gamma_0$ and $\Gamma_1$ be  maximal chains in the intervals $[\emptyset,\sodeal P a]_{L}$ and $[\sideal P a,P]_{L}$ of 
$\alg F=L$, respectively. Then $\Gamma:=\Gamma_0\cup\Gamma_1$ is a maximal chain of $L$. 
Observe that  $[\emptyset,\sodeal P a]_{L}\cap\alg N=\emptyset$ since the members of $\alg N$ contain $a$. Also, $[\sideal P a,P]_{L}\cap\alg N=\emptyset$ since otherwise \eqref{eq:kMnTszLGnCskHvgrv} would imply that $\sodeal P a\in \alg T$, contradicting $\sodeal P a\fprec\sideal P a\ni a$. Hence, $\Gamma$ is also a maximal chain of $K$ and $\length K=\length L$. By Lemma \ref{lemma:nvRgJzsj}, $L$ is a sublattice of $K$. 
Finally, $\Jir K\cong Q$ by Theorem~\ref{thmfaigle}(D), whence $k:=|\set{(x,y): x\leq_{\Jir K} y }|=|\set{(x,y): x\leq_{Q} y }| < |\set{(x,y): x\leq_{P} y }|=|\set{(x,y): x\leq_{\Jir L} y }|$. If $k\neq 0$, then we repeat the construction with $K$ in place of $L$. In a finite number of steps, $k$  becomes 0; then $K$ is a geometric lattice, completing the proof of Proposition~\ref{prop:ggkss}.
\end{proof}

\section{Some definitions and proving Theorem~\ref{thm:gknapp}}
\label{sect:defproof}

\optional{The \emph{width}  of a finite poset $P$ is the maximum of the sizes of its antichains. By  Dilworth \cite[Theorem  1.1]{dilworth}, $\width P$ is the least $n\in\Nplu:=\set{1,2,3\dots}$ such that $P$ is  the union of $n$ (not necessarily disjoint) chains.
A lattice $L$ is \emph{slim} if it is finite and $\width{\Jir L}\leq 2$. Note that slim lattices are planar; see Cz\'edli and Schmidt~\cite[Proposition 5]{czgschtvisual}. Following \cite{czgschtvisual}, a finite semimodular lattice $L$ is said to be a \emph{slim rectangular lattice} if $\Jir L$ is the union of two chains $C$ and $D$ such that for all $c\in C$ and $d\in D$, the elements $c$ and $d$ are incomparable; chains are nonempty by definition. Our definition of slim semimodular lattices and that of slim rectangular lattices are non-visual and do not refer to (Hasse) diagrams. We know from Cz\'edli and Schmidt~\cite{czgschtvisual} that these definitions are equivalent to the visual original ones  
given by Gr\"atzer and Knapp~\cite{gratzerknapp1} and \cite{gratzerknapp3}. 
}

\optional{
\begin{figure}[h]
\centerline
{\includegraphics[width=\textwidth]{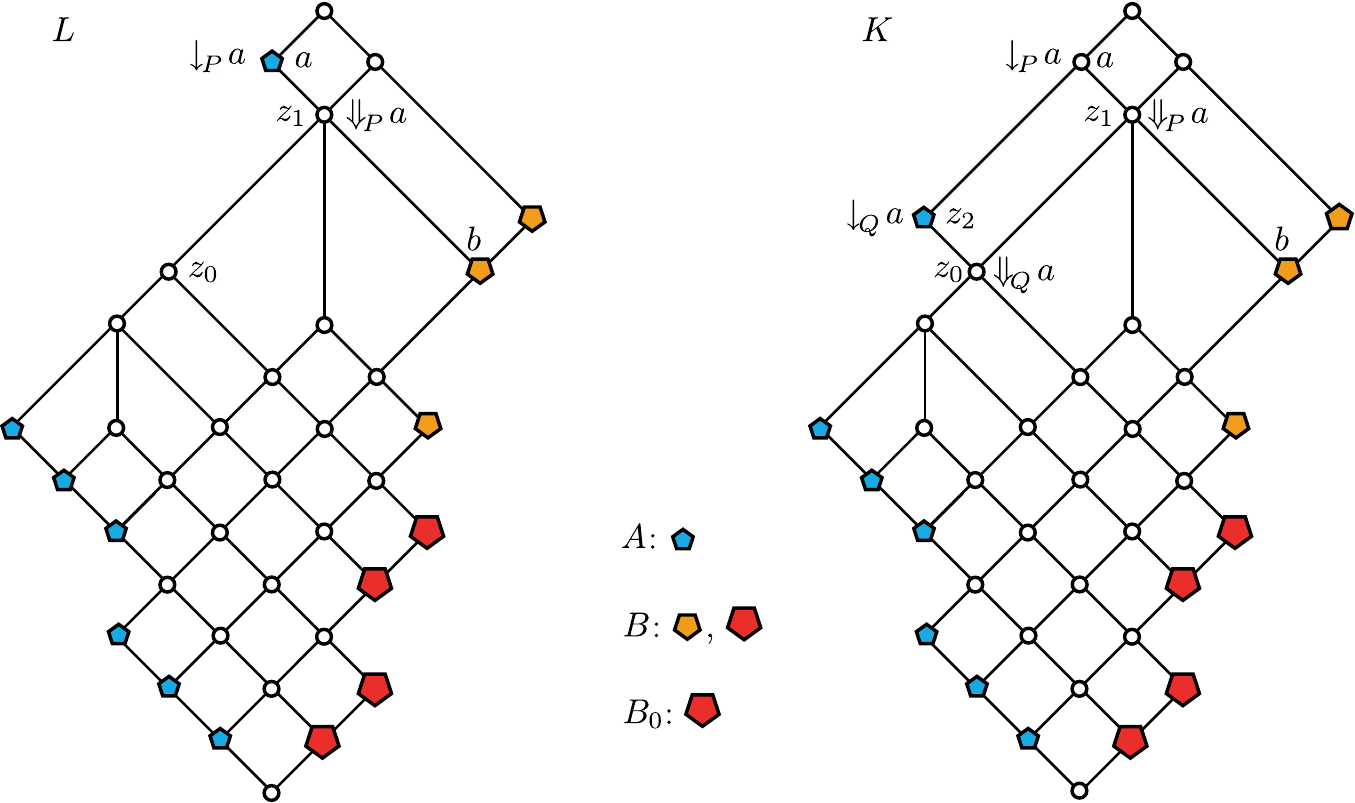}}      
\caption{An illustration for Section~\ref{sect:defproof}}\label{fig2}
\end{figure}
}

\begin{proof}[Proof of Theorem~\ref{thm:gknapp}]
Let $M_1 \leq M_2\leq M_3$ be lattices where $\leq$ stands for ``sublattice of''. The rule  $\rho_{M_2,M_1}\circ\rho_{M_3,M_2}=\rho_{M_3,M_1}$ shows that ``congruence-preserving extension'' is a transitive relation among lattices. This allows us to prove the theorem ``step-by-step'', getting closer to a slim rectangular lattice at each step. 
Let $L$ be a slim semimodular lattice with $|L|\geq 3$. As the first step, we can assume that $L$ is not a chain since otherwise Lemma~\ref{lemma:corner} allows us to replace $L$ by an $(|L|+1)$-element congruence-preserving extension of the same length, which is not a chain but slim and semimodular (in fact, distributive). 
Since $L$ is slim and it is not a chain, $\Jir L$ is the union of two \emph{disjoint} chains, $C$ and $D$. (Replace $D$ by $D\setminus C$ if necessary.)  Let $\delta_{C,D}:=|\set{(c,d)\in C\times D: c\nonparallel d}|$, where $\nonparallel$ means ``comparable'',  and $\delta(L):=\min\{\delta_{C,D}:  C$  and $D$ are disjoint chains with $C\cup D=\Jir {L}\}$. Since $\delta(L)=0$ if and only if $L$ is rectangular, it suffices to find a slim semimodular congruence-preserving extension $K$ of $L$ such that $\delta(K)<\delta(L)$, if $\delta(L)>0$.

Assume that $\delta(L)>0$ and this is witnessed by disjoint chains $A$ and $B$ in $\Jir L$. Then $\Jir L=A\cup B$ and $\delta(A,B)=\delta (L)$. Let $\Geom L=\wha F=(P,\alg F)=(\Jir L,\alg F)$  be the Faigle geometry associated with $L$; see \eqref{eq:mtcnzsVMfgTpnLnH}. Since $A$ and $B$ play a symmetrical role, we can pick an $a\in A$ and a $b\in B$ such that $a$ covers $b$ in $P$, in notation, $b\prec_P a$. Understanding the $\bigvee$ and $<$ in $L$, we let 
\begin{align}
&z_1:=\bigvee \sodeal P a,\quad   B_0:=\set{y\in B: y\vee \bigvee \sodeal A a < z_1},\,\,\text{ and observe that}
\label{eq:mprszHtzPrcnhKfSn}\\
&\sodeal P a = \sodeal A a\cup\sideal B b\qquad \text{ and }\qquad b\notin B_0. 
\label{eq:wqGrmTngkmgrRlcrD}
\end{align}
On the underlying set of $P=\Jir L$, we define a new poset $Q=(\Jir L,\leq_Q)$ by letting
\begin{equation}y<_Q x \defiff \left\{
\parbox{6cm}{
$y<_P x\neq a$, \quad or \\
$y<_P x = a$ \text{ and } $y\in B_0\cup \sodeal A a$.}\right.
\label{pbxDfQrrh}
\end{equation}
Since $B_0\subseteq \sideal P a$,  ``$y<_P$'' can be omitted  from ``$y<_P x = a$''. Using that $B_0\cup\sodeal A a$ is a down-set of $P$, it follows that $\leq_Q$ is a (partial) ordering. Clearly, for $d\in P\setminus\set a$, 
\begin{align}
&\text{$\sideal Q d=\sideal P d$, \ 
$\sodeal Q d=\sodeal P d$, \ $\sideal Q a\subset\sideal P a$, \ and  $\sodeal Q a\subset\sodeal P a$.}
\label{txt:nmTszbZwkc}\\
&\parbox{7.7cm}{Note that $A$ and $B$ are chains both in $Q$ and $P$. Also, ``$\restrict {\leq_P}A$'' = ``$\restrict {\leq_Q}A$", and ``$\restrict {\leq_P}B$'' = ``$\restrict {\leq_Q}B$".}
\label{pbx:pPnvTnmRmHthSz}
\end{align}
Hence, say, $\sodeal A a$ makes sense. 
We define the required  lattice as $K:=\Lat{\wha G}$ where 
\begin{equation}
\alg G:= \alg F\cup\set{\sideal Q a} 
\qquad\text{ and }\qquad \wha G:=(Q,\alg G).
\label{eq:htlBmfnNmkdLkzstRpt}
\end{equation}
Of course, we need to show that $\wha G$ is a Faigle geometry. We claim that 
\begin{equation}
\parbox{9.4cm}{$
\sideal Q a=\sideal A a\cup B_0$,  $\,\,\sodeal Q a=\sodeal A a\cup B_0\in \alg F$,  $\,\,z_0:=\bigvee\sodeal Q a\prec_L z_1$, and
$z_1$ is the only cover of $z_0$ in $L$.
}
\label{eq:hfmHsTclkTmhR}
\end{equation}
The equalities in  \eqref{eq:hfmHsTclkTmhR} are clear by  \eqref{pbxDfQrrh}. Since $B_0$ is a (possibly empty) subset of a chain, ``$<$'' in place of ``$\prec_L$'' follows from \eqref{eq:mprszHtzPrcnhKfSn}. For any $X'\in[\sodeal Q a, \sodeal P a]_{\alg F}$, we have that $X'\cap A=\sodeal A a$ by \eqref{eq:wqGrmTngkmgrRlcrD} and the equalities in \eqref{eq:hfmHsTclkTmhR}, whence ``$\prec_L$''  follows from \eqref{eq:mprszHtzPrcnhKfSn}, and we also obtain that $z_1$ is the only cover of $z_0$ in $L$. 
By \eqref{eq:mtcnzsVMfgTpnLnH}, to obtain that 
$\sodeal Q a=\sodeal A a\cup B_0\in\alg F$, it suffices to show that $\sodeal Q a=P\cap\sideal L{z_0}$. 
 The inclusion $\sodeal Q a\subseteq P\cap\sideal L {z_0}$ is clear. 
Conversely, let $c\in  P\cap\sideal L {z_0}$. Then $c=a'\in A$ or $c=b'\in B$. By \eqref{txt:nmTszbZwkc}, $z_0 <\bigvee \sideal P a=a$ since $a\in\Jir L$. Hence, $a'\leq z_0$ implies that $a'\in\sodeal A a\subseteq \sodeal Q a$. By \eqref{eq:mprszHtzPrcnhKfSn}, $b'\leq z_0<z_1$ gives that $b'\in B_0 \subseteq \sodeal Q a$, proving \eqref{eq:hfmHsTclkTmhR}.

It follows from \eqref{txt:nmTszbZwkc} and \eqref{eq:hfmHsTclkTmhR} that $\wha G$ satisfies (F$\dwn$) and (Pr) since so does $\wha F$. For $X\in\alg F$, if $a\in X$, then $\sideal Q a\subseteq \sideal P a\subseteq X$ shows that $X\cap \sideal Q a=\sideal Q a\in\alg G$.
If $a\notin X$, then \eqref{eq:hfmHsTclkTmhR} gives that 
$X\cap \sideal Q a=X\cap \sodeal Q a\in\alg F\subseteq \alg G$. Thus, $\wha G$ satisfies (F$\cap$) since so does $\wha F$. 
In particular, $K$ is a lattice. We claim that $L$ is a cover-preserving subposet of $K$. For the sake of contradiction, suppose the contrary. 
Then, since $\alg G\setminus\alg F=\set{\sideal Q a}$,  $X_1\fprec X_2$ for some $X_1,X_2\in\alg F$ but 
$X_1\subset \sideal Q a \subset X_2$. From $X_1\subset\sideal Q a$ and \eqref{eq:hfmHsTclkTmhR}, we have that $X_1\subseteq \sodeal Q a\in\alg F$. Since (F$\dwn$) and (Pr) hold in $\wha F$ and $a\in \sideal Q a \subset X_2$, we have that $\sideal P a\subseteq X_2$. 
Since $b\in\sodeal P a\setminus\sodeal Q a$,  we obtain that $\sodeal Q a\subset \sodeal P a\in\alg F$. So $X_1\subseteq \sodeal Q a \subset \sodeal P a\subset \sideal P a\subseteq X_2$ with $\set{\sodeal Q a, \sodeal P a,\sideal P a}\subseteq \alg F$, contradicting that $X_1\fprec X_2$. Thus, $L$ is a cover-preserving subposet of $K$.

Next, to show  that $\alg G$ satisfies (CP), assume that $u\in Q$, $X\in\alg G$, $u\notin X$, and $\sodeal Q u\subseteq X$. We need to find a $Y\in \alg G$ such that $X\gprec Y\ni u$. There are three cases.

First, if $X\in\alg F$ and  $u\neq a$, then $\sodeal P u=\sodeal Q u\subseteq X$ by \eqref{txt:nmTszbZwkc}, so there is a $Y\in \alg F$ with $X\fprec Y\ni u$. Then  $X\gprec Y\ni u$ since $L$ is a cover-preserving subposet of $K$.

Second, let $X\in \alg F$ and $u=a$. Then  $\sodeal Q a\subseteq X$. If $X=\sodeal Q a$, then we can clearly let $Y:=\sideal Q a$. So we can assume that $\sodeal A a\cup B_0 = \sodeal Q a\subset  X$.
Since $a=u\notin X$ and $X$ is a down-set in $P$, $X\cap\sfilter P a=\emptyset$. Hence, there is a down-set
$B'$ of $B$  such that  $B'\not\subseteq B_0$ and $X=\sodeal Q a \cup B'$. The down-sets of the chain $B$ are comparable, whence  $B_0\subset B'$. Hence \eqref{eq:hfmHsTclkTmhR} and \eqref{eq:mprszHtzPrcnhKfSn} give that $X=\sodeal A a\cup B'$ and  $z_1\leq \bigvee X$. Thus,  Theorem~\ref{thmfaigle} yields that $\sodeal P a\subseteq X$. Since $\wha F$ satisfies (CP), we obtain a $Y\in \alg F$ with $X\fprec Y\ni a=u$, and $X\gprec Y$ as  $L$ is a cover-preserving subposet of $K$.

Third, let $X\in \alg G\setminus\alg F$, that is, $X=\sideal Q a$.
Since $u\notin X$ but $\sideal A a\subseteq X$, we have that 
 $u\notin \sideal A a$. Also,  $u\notin\sfolter A a$ since otherwise $b\in\sideal P a\subseteq \sodeal P u=\sodeal Q u\subseteq X=\sideal Q a$ would be a contradiction. Hence, $u\notin A$ and so $u\in B$. If we had that $u>_B b$, then $b\in\sodeal Q u\subseteq X=\sideal Q a$ would be the same contradiction as before. 
Hence, $u\in\sideal B b\subseteq\sideal P a=:Y$. Clearly, $\sideal Q a\gprec \sideal P a$. This completes the third case. Thus, $\wha G$ is a Faigle geometry. Hence, $K$ is semimodular. Since $L$ is a cover-preserving subposet of $K$, $\length K=\length L$. Trivially, $L$ is a meet-subsemilattice of $K$. Thus, Lemma~\ref{lemma:nvRgJzsj} implies that $L$ is a sublattice of $K$.

Implicitly, the last part of the proof often uses the canonical correspondence formulated in Theorem~\ref{thmfaigle}. 
From  \eqref{eq:mprszHtzPrcnhKfSn}, \eqref{eq:hfmHsTclkTmhR}, 
$|\sideal P a\setminus\sodeal P a|=1$, and the fact that $L$ is a cover-preserving sublattice of $K$,  we obtain that $z_0\prec_K z_1\prec_K a$. Let  $z_2:=\sideal Q a\in K$. Since $\sideal P a$ is clearly the only cover of $\sideal Q a$ in $\alg G$, $a$ is the only cover of $z_2$ in $K$. Thus,  $z_2\in\Mir K$. Also, $\sodeal Q a$ is the only lower cover of $\sideal Q a$, whence $z_0\prec_K z_2$ and  $z_2\in\Jir K$. So $z_2\in \Jir K\cap\Mir K$.  We know that $a\in \Jir L$ while  \eqref{eq:hfmHsTclkTmhR} gives that    $z_0\in\Mir L$. By these facts, Lemma~\ref{lemma:corner} applies with $(z_0,z_1,z_2,a)$  playing the role of $(a,c,d,b)$ and completes the proof of Theorem~\ref{thm:gknapp}.
\end{proof}


\color{black!60!green}

\section{Appendix: bibliography of slim or planar semimodular lattices}
\label{sect:addbibl}

To help future research and to serve as a ``reference paper'',
we present a list\footnote{See \quad \texttt{www.math.u-szeged.hu/\textasciitilde{}czedli/m/listak/publ-psml.pdf} \quad for possible updates}
of publications on the class of planar semimodular lattices, where the lattices occurring in Theorem~\ref{thm:gknapp} belong.
The previous sections make no direct reference to this ``additional bibliography'' below, which consists of those publications that 
\begin{itemize}
\item deal with slim semimodular or planar semimodular lattices, or
\item generalize, use, or enumerate slim semimodular lattices.
\end{itemize}
Note that in the study of planar semimodular lattices, the slim ones and the slim rectangular lattices have always played a distinguished role.

\renewcommand{\refname}{Additional bibliography}

\makeatletter
\renewcommand\@biblabel[1]{[+#1]}
\makeatother


\renewcommand{\refname}{References}

\makeatletter
\renewcommand\@biblabel[1]{[#1]}
\makeatother

\color{black}


\begin{thebibliography}{99}
\bibliographystyle{plain}


\bibitem{AB-adariczg} 
 Adaricheva, K., Cz\'edli, G.:
 Note on the description of join-distributive lattices by permutations. Algebra Universalis \tbf{72}, 155--162 (2014)



\bibitem{AB-czg96}   
    Cz\'edli, G.: 
    The matrix of a slim semimodular lattice, Order 29 (2012) 85--103

\bibitem{duplumczgreprhomo98}   
    Cz\'edli, G.: 
   Representing homomorphisms of distributive lattices as restrictions of congruences of rectangular lattices, Algebra Universalis 67, 313--345 (2012) \dupl{czgreprhomo98}

\bibitem{AB-czgcircles}   
    Cz\'edli, G.: 
    Finite convex geometries of circles.
 Discrete Mathematics \tbf{330},  61--75 (2014) 


\bibitem{AB-czgcoord}   
    Cz\'edli, G.:
    Coordinatization of finite join-distributive lattices.
    Algebra Universalis \tbf{71}, 385--404 (2014)     

\bibitem{AB-czg111}
  Cz\'edli, G.: 
  Patch extensions and trajectory colorings of slim rectangular lattices, Algebra Universalis \tbf{72},  125--154 (2014)

\bibitem{AB-czg112}  
    Cz\'edli, G.:
 A note on congruence lattices of slim semimodular lattices, Algebra Universalis, \tbf{72}, 225--230 (2014)

\bibitem{AB-czgqplanar}  
    Cz\'edli, G.: 
    Quasiplanar diagrams and slim semimodular lattices.
    Order \tbf{33}, 239--262  (2016) 

\bibitem{AB-czg124}   
  Cz\'edli, G.: 
  The asymptotic number of planar, slim, semimodular lattice diagrams, Order \tbf{33}, 231--237  (2016)

\bibitem{AB-czg132}
  Cz\'edli, G.:
  Diagrams and rectangular extensions of planar semimodular lattices. Algebra Universalis \tbf{77}, 443--498 (2017)

\bibitem{AB-czglamps} 
    Cz\'edli, G.: Lamps in slim rectangular planar semimodular lattices; Acta Sci. Math. (Szeged), DOI 10.14232/actasm-021-865-y \red{(not functioning yet)}, 
\texttt{http://arxiv.org/abs/2101.02929} 


\bibitem{duplumczg-slimpatchabsretr}
   Cz\'edli, G.:
   Slim patch lattices as absolute retracts and maximal lattices.
   \texttt{http://arxiv.org/abs/2105.12868} \dupl{czg-slimpatchabsretr}


\bibitem{AB-czgaxiombipart}  
    Cz\'edli, G.: 
    Cyclic congruences of slim semimodular lattices and non-finite axiomatizability of some finite structures. \quad 
    \texttt{http://arxiv.org/abs/2102.00526} 


\bibitem{AB-czgCurrentPaper}
  Cz\'edli, G.:
Revisiting Faigle geometries from a perspective of semimodular lattices. 
(The present paper)



\bibitem{AB-czgdgyk}  
  Cz\'edli, G., D\'ek\'any, T., Gyenizse, G., Kulin, J.:
  The number of slim rectangular lattices. 
  Algebra Universalis \tbf{75}, 33--50 (2016) 

\bibitem{AB-czgdoszu117}
 Cz\'edli, G., D\'ek\'any, T., Ozsv\'art, L., Szak\'acs, N., Udvari, B.:
 On the number of slim, semimodular lattices, Mathematica Slovaca, \tbf{66}, 5--18 (2016)


\bibitem{duplumczgggltsta}  %
   Cz\'edli, G.,  Gr\"atzer, G.:
   Planar semimodular lattices: structure and diagrams. Chapter 3  in: Gr\"atzer, G.,
Wehrung, F. (eds.), Lattice Theory: Special Topics and Applications, pp 91--130,  Birkh\"auser, Basel (2014) 
\dupl{czgggltsta}

\bibitem{AB-czgggresections}  
   Cz\'edli, G.,  Gr\"atzer, G.:
    Notes on planar semimodular lattices. VII. Resections of planar semimodular lattices.
    Order \tbf{30}, 847--858  (2013) 

\bibitem{AB-czgginprepar}    
   Cz\'edli, G.,  Gr\"atzer, G.: 
   A new property of congruence lattices of slim, planar, semimodular lattices.  \quad
\texttt{http://arxiv.org/abs/2103.04458} 

\bibitem{AB-czggghlswing}   
   Cz\'edli, G.,  Gr\"atzer, G., Lakser, H.:
   Congruence structure of planar semimodular lattices: the general swing lemma. 
   Algebra Universalis \tbf{79}:40,  18 pp (2018) 

\bibitem{duplumczgkurusa}  
   Cz\'edli, G., Kurusa, \'A.:
   A convex combinatorial property of compact sets in the plane and its roots in lattice theory. 
   Categories and General Algebraic Structures with Applications \tbf{11}, 57--92 (2019)\quad
\texttt{http://cgasa.sbu.ac.ir/article\textunderscore{}82639.html}
\dupl{czgkurusa}

\bibitem{AB-czgmakay}  
   Cz\'edli, G., Makay, G.:
   Swing lattice game and a direct proof of the swing lemma for planar semimodular lattices.
    Acta Sci. Math. (Szeged) \tbf{83}, 13--29 (2017) 

\bibitem{AB-czgmolkhasi} 
  Cz\'edli, G., Molkhasi, A.:
  Absolute retracts for finite distributive lattices and slim semimodular lattices.
  \texttt{http://arxiv.org/abs/2105.10604}  

\bibitem{AB-czgou104}
 Cz\'edli, G., Ozsv\'art, L., Udvari, B.:
 :How many ways can two composition series intersect? Discrete Mathematics 312, 3523--3536 (2012) 

\bibitem{AB-czgscht91}
   Cz\'edli, G., Schmidt, E.T.:
  Some results on semimodular lattices, Contributions to General Algebra 19. Proceedings of the Olomouc Conference 2010 (AAA 79+ CYA 25) , Verlag Johannes Hein, Klagenfurt 2010, 45-56. ISBN 978-3-7084-0407-3

\bibitem{AB-czgschtJH}     
   Cz\'edli, G., Schmidt, E.T.:
   The Jordan-H\"older theorem with uniqueness for groups and semimodular lattices. 
   Algebra Universalis  \tbf{66}, 69--79 (2011) 


\bibitem{AB-czgscht108}
   Cz\'edli, G., Schmidt, E.T.:
   Composition series in groups and the structure of slim semimodular lattices, Acta Sci Math. (Szeged) 79 (2013), 369--390.

\bibitem{duplumczgschtvisual}    
   Cz\'edli, G., Schmidt, E.T.:
   Slim semimodular lattices. I. A visual approach.
   Order \tbf{29}, 481--497 (2012) \dupl{czgschtvisual}


\bibitem{AB-czgschtpatchwork}  
 Cz\'edli, G.,  Schmidt, E.T.:
 Slim semimodular lattices. II. A description by patchwork systems.
 Order \tbf{30}, 689--721 (2013) 

\bibitem{AB-gyenkul}
  D\'ek\'any, T., Gyenizse, G., Kulin, J.:
  Permutations assigned to slim rectangular lattices. Acta Sci. Math. (Szeged) \tbf{82}, 19--28 (2016)



\bibitem{AB-ggBsCnr}
  Gr\"atzer, G.:
    Planar semimodular lattices: congruences. in: Gr\"atzer, G.,
Wehrung, F. (eds.), Lattice Theory: Special Topics and Applications. Vol. 1, Chapter 4, pp 131--165,  Birkh\"auser, Basel (2014)

\bibitem{AB-ggONgczresult}
 Gr\"atzer, G.:
 On a result of G\'abor Cz\'edli concerning congruence lattices of planar semimodular lattices. Acta Sci. Math. (Szeged) \tbf{81}, 25--32  (2015) 


\bibitem{AB-ggVI}
  Gr\"atzer, G.: 
  Notes on planar semimodular lattices. VI. On the structure theorem of planar semimodular lattices. Algebra Universalis \tbf{69}, 301--304 (2013)

\bibitem{AB-ggswinglemma}   
 Gr\"atzer, G.: 
 Congruences in slim, planar, semimodular lattices: The Swing Lemma. 
 Acta Sci. Math. (Szeged) \tbf{81}, 381--397 (2015)

\bibitem{AB-gg:conforkext}
  Gr\"atzer, G.:
  Congruences of fork extensions of slim, planar, semimodular lattices. Algebra Universalis \tbf{76}, 139--154 (2016)


\bibitem{AB-ggCTpsl}
  Gr\"atzer, G.: 
  Congruences and trajectories in planar semimodular lattices. Discuss. Math. Gen. Algebra Appl. \tbf{38}, 131--142 (2018)

\bibitem{AB-ggVIII}
  Gr\"atzer, G.: 
  Notes on planar semimodular lattices. VIII. Congruence lattices of SPS lattices. Algebra Universalis \tbf{81}:15 (3 pp)(2020)


\bibitem{AB-ggapplczgschtseq}  
    Gr\"atzer, G.: Applying the Cz\'edli-Schmidt sequences to congruence properties of planar semimodular lattices. Discuss. Math. Gen. Algebra Appl. \tbf{41}, 153--169 (2021) 

\bibitem{AB-ggC1dgrms}  
  Gr\"atzer, G.:
  Notes on planar semimodular lattice. IX. $\mathcal C_1$-diagrams.
  Discuss. Math. Gen. Algebra Appl., to appear.
\texttt{https://arxiv.org/abs/2104.02534}

\bibitem{AB-ggswingC1dgr}  
 Gr\"atzer, G.: Using the swing lemma and  $\mathcal C_1$-diagrams for congruences of planar semimodular lattices.
  \texttt{https://arxiv.org/abs/2106.03241.pdf}

\bibitem{duplumgratzerknapp1}  
   Gr\"atzer, G., Knapp, E.:
   Notes on planar semimodular lattices. I.  Construction. 
   Acta Sci.\ Math.\ (Szeged) \tbf{73}, 445--462 (2007) \dupl{gratzerknapp1}

\bibitem{AB-gratzerknappnn}  
   Gr\"atzer, G., Knapp, E.:
   A note on planar semimodular lattices. Algebra Universalis \tbf{58},  497--499 (2008)

\bibitem{AB-gratzerknapp2}  
   Gr\"atzer, G., Knapp, E.:
   Notes on planar semimodular lattices. II. Congruences. Acta Sci. Math. (Szeged) \tbf{74}, 37--47 (2008)

\bibitem{duplumgratzerknapp3}  
   Gr\"atzer, G., Knapp, E.:
   Notes on planar semimodular lattices. III. Congruences of rectangular lattices. 
   Acta Sci. Math. (Szeged), \tbf{75}, 29--48 (2009) \dupl{gratzerknapp3}

\bibitem{AB-gratzerknapp4}  
   Gr\"atzer, G., Knapp, E.:
  Notes on planar semimodular lattices. IV. The size of a minimal congruence lattice representation with rectangular lattices. Acta Sci. Math. (Szeged) \tbf{76}, 3--26 (2010)


\bibitem{AB-ggschtExtThm}
  Gr\"atzer, G., Schmidt, E. T.:
  An extension theorem for planar semimodular lattices. Period. Math. Hungar. \tbf{69}, 32--40 (2014)

\bibitem{AB-ggschtshrTrp}
  Gr\"atzer, G., Schmidt, E. T.:
A short proof of the congruence representation theorem of rectangular lattices. Algebra Universalis \tbf{71}, 65--68 (2014)


\bibitem{AB-ggWares}
  Gr\"atzer, G., Wares, T.::
Notes on planar semimodular lattices. V. Cover-preserving embeddings of finite semimodular lattices into simple semimodular lattices. Acta Sci. Math. (Szeged) \tbf{76},  27--33 (2010)

\end{thebibliography}

\begin{thebibliography}{99}



\bibitem{crawleydilworth} 
  P. Crawley, R.P. Dilworth: 
  Algebraic Theory of Lattices. 
  Prentice Hall, Englewood Cliffs, NJ, 1973

\bibitem{czgreprhomo98}  
    Cz\'edli, G.: 
   Representing homomorphisms of distributive lattices as restrictions of congruences of rectangular lattices, Algebra Universalis 67, 313--345 (2012)


\bibitem{czg-slimpatchabsretr} 
   Cz\'edli, G.:
   Slim patch lattices as absolute retracts and maximal lattices.
   \texttt{http://arxiv.org/abs/2105.12868}



\bibitem{czgggltsta} 
   Cz\'edli, G.,  Gr\"atzer, G.:
   Planar semimodular lattices: structure and diagrams. Chapter 3  in: Gr\"atzer, G.,
Wehrung, F. (eds.), Lattice Theory: Special Topics and Applications, pp 91--130,  Birkh\"auser, Basel (2014)


\bibitem{czgkurusa} 
   Cz\'edli, G., Kurusa, \'A.:
   A convex combinatorial property of compact sets in the plane and its roots in lattice theory. 
   Categories and General Algebraic Structures with Applications \tbf{11}, 57--92 (2019)\quad
\texttt{http://cgasa.sbu.ac.ir/article\textunderscore{}82639.html}



\bibitem{czgscht2geom} 
 Cz\'edli, G., Schmidt, E.T.:
 A cover-preserving embedding of semimodular lattices
into geometric lattices. 
 Advances in Mathematics \tbf{225}, 2455--2463 (2010)


\bibitem{czgschtvisual} 
   Cz\'edli, G., Schmidt, E.T.:
   Slim semimodular lattices. I. A visual approach.
   Order \tbf{29}, 481--497 (2012)



\bibitem{dilworth} 
 Dilworth, R.P.:
 A decomposition theorem for partially ordered sets.
 Ann. of Math. \tbf{(2) 51}, 161--166 (1950)


\bibitem{faigle} 
  Faigle, U.:
  Geometries on partially ordered sets.
  J. Combinatorial Theory B \tbf{28}, 26--51 (1980)


\bibitem{ggkiss} 
  Gr\"atzer, G., Kiss, E. W.:
  A construction of semimodular lattices. 
  Order \tbf{2}, 351--365 (1986)

\bibitem{gratzerknapp1} 
   Gr\"atzer, G., Knapp, E.:
   Notes on planar semimodular lattices. I.  Construction. 
   Acta Sci.\ Math.\ (Szeged) \tbf{73}, 445--462 (2007)

\bibitem{gratzerknapp3} 
   Gr\"atzer, G., Knapp, E.:
   Notes on planar semimodular lattices. III. Congruences of rectangular lattices. 
   Acta Sci. Math. (Szeged), \tbf{75}, 29--48 (2009)


\bibitem{gr-scht} 
   Gr\"atzer, G., Schmidt, E. T.:
   The strong independence theorem for automorphism groups and congruence lattices of finite lattices.
   Beitr\"age Algebra Geom. \tbf{36}, 97--108 (1995)

\bibitem{kellyrival} 
  Kelly, D., Rival, I.:
  Planar lattices.
  Canadian J. Math. \tbf{27}, 636--665 (1975)


\bibitem{qbushMR} 
 Quackenbush, R. W.: Review on Faigle~\cite{faigle}. MathSciNet, MR565509 (81m:05054)

\bibitem{skublics}
  Skublics, B.: 
  Isometrical embeddings of lattices into geometric lattices. Order 30, 797--806 (2013)

\bibitem{stern} 
  Stern, M.: Semimodular Lattices --- Theory and Application, Cambridge University Press, 1999


\bibitem{wild} 
 Wild, M.: 
 Cover preserving embedding of modular lattices into partition lattices.
 Discrete Math. \tbf{112}, 207--244 (1993)
 
  
\end{thebibliography}
\end{document}